\documentclass[10pt, a4paper]{amsart}
\usepackage{amsmath,amsfonts,amsthm,amssymb}
\usepackage{mathtools,braket,enumitem}
\usepackage[varg]{txfonts}
\usepackage{ascmac} 
\usepackage{color}
\usepackage{bm} 
\usepackage{setspace} 
\usepackage{multicol} 
\usepackage{graphicx}
\usepackage{multirow} 
\usepackage[all]{xy} 
\usepackage{hyperref}
\hypersetup{
  colorlinks   = true, 
  urlcolor     = blue, 
  linkcolor    = blue, 
  citecolor   = red 
}
\usepackage{tikz}
\usetikzlibrary{intersections, calc, arrows.meta, graphs}
\tikzstyle{v} = [circle, draw, inner sep=2pt, minimum size=3pt, fill=black]

\allowdisplaybreaks

\usepackage{comment}

\newcommand{\ceq}{\coloneqq}

\newcommand{\GAMMA}{\mathit{\Gamma}}

\newcommand{\Gx}{{\mathit{\Gamma}(x)}}
\newcommand{\G}{\GAMMA} 
\newcommand{\g}{\gamma} 
\newcommand{\X}{L} 
\newcommand{\subG}{{\G'}} 

\DeclareMathOperator{\id}{id}
\DeclareMathOperator{\rank}{rank}
\DeclareMathOperator{\lcm}{lcm}
\DeclareMathOperator{\Hom}{Hom}
\DeclareMathOperator{\GL}{GL}
\DeclareMathOperator{\mmoodd}{mod}

\DeclareMathOperator{\vol}{vol}
\DeclareMathOperator{\stt}{R}
\newcommand{\st}{{\stt}}

\DeclareMathOperator{\tr}{tr}

\DeclareMathOperator{\image}{im}
\DeclareMathOperator{\diag}{diag}
\DeclareMathOperator{\LL}{L}
\DeclareMathOperator{\Ind}{Ind}
\DeclareMathOperator{\Res}{Res}

\theoremstyle{plain} 
\newtheorem{theorem}{Theorem}[section]
\newtheorem{proposition}[theorem]{Proposition}
\newtheorem*{proposition*}{Proposition}

\newtheorem*{claim*}{Claim}
\newtheorem{lemma}[theorem]{Lemma}
\newtheorem{corollary}[theorem]{Corollary}

\theoremstyle{definition} 
\newtheorem*{definition*}{Definition}

\newtheorem{example}[theorem]{Example}

\newtheorem*{remark*}{Remark}

\numberwithin{equation}{section}

\title[The quasi-polynomiality of $\mmoodd{q}$ permutation representations]{The quasi-polynomiality of $\mmoodd{q}$ permutation representations for a linear finite group action on a lattice}
\author[R. Uchiumi and M. Yoshinaga]{Ryo Uchiumi and Masahiko Yoshinaga}
\date{August 30, 2024}
\thanks{R. U. was supported by JST SPRING, Grant Number JPMJSP2138.}
\thanks{M. Y. was partially supported by JSPS KAKENHI, Grant Number JP23H00081.}

\begin{document}

\begin{abstract}
For given linear action of a finite group on a lattice and a positive integer $q$, 
we prove that the $\mmoodd{q}$ permutation representation is a quasi-polynomial 
in $q$. Additionally, we establish several results that can be considered as 
$\mmoodd{q}$-analogues of results by Stapledon for equivariant 
Ehrhart quasi-polynomials. 
We also prove a reciprocity-type result for multiplicities of irreducible decompositions. 
\end{abstract}

\maketitle

\tableofcontents



\section{Introduction}
\subsection{Quasi-polynomials}
Let $R$ be a commutative ring. A function $f : \mathbb{Z}_{(>0)} \to R$ is called a \textit{quasi-polynomial} if there exist a positive integer $\tilde{n} \in \mathbb{Z}_{>0}$ and polynomials $g_1(t),\ldots,g_{\tilde{n}}(t) \in R[t]$ such that 
\begin{align*}
f(q) = g_r(q),\quad \text{if $q \equiv r \mod{\tilde{n}}$}\qquad (1 \leq r \leq \tilde{n}).
\end{align*}
The positive integer $\tilde{n}$ is called a \textit{period} and each polynomial $g_r$ is called the \textit{constituent} of $f$.
The quasi-polynomial $f$ has degree $d$ if all the constituents have degree $d$.
Moreover, the quasi-polynomial $f$ has the \textit{$\gcd$-property} if the polynomial $g_r$ depends on $r$ only through $\gcd\{\tilde{n},\, r\}$. In other words, $g_{r_1} = g_{r_2}$ if $\gcd\{\tilde{n},\, r_1\} = \gcd\{\tilde{n},\, r_2\}$.

Quasi-polynomials play important roles in many areas of mathematics. 
They appear frequently as counting functions 
(in this case, $R = \mathbb{Z}$).
In particular, the following two notions have been actively studied. 

\begin{example}[The Ehrhart quasi-polynomial]
Let $\mathcal{P}$ be a rational polytope in $\mathbb{R}^n$.  
For $q\in\mathbb{Z}_{\geq0}$, define 
\[
\LL_\mathcal{P}(q) \ceq \#(q\mathcal{P} \cap \mathbb{Z}^n). 
\]
Then $\LL_{\mathcal{P}}(q)$ is a quasi-polynomial ({\cite[Theorem 3.23]{BR:2007}}), 
known as the Ehrhart quasi-polynomial. 
\end{example}

\begin{example}[The characteristic quasi-polynomial]
Let $L \simeq \mathbb{Z}^\ell$ be a lattice and $L^\vee \ceq \Hom_{\mathbb{Z}}
(L,\mathbb{Z})$ be the dual lattice.
Given $\alpha_1,\ldots,\alpha_n \in L^\vee$, 
we can associate a hyperplane arrangement $\mathcal{A} \ceq \set{H_1,\ldots,H_n}$ in $\mathbb{R}^\ell \simeq L \otimes \mathbb{R}$, where 
\begin{align*}
H_i \ceq \set{ x \in L\otimes \mathbb{R} \mid \alpha_i(x) = 0}.
\end{align*}
For a positive integer $q \in \mathbb{Z}_{>0}$, define the $\mmoodd{q}$ complement 
of the arrangement by 
\begin{align*}
\begin{split}
M(\mathcal{A};\,q) &\ceq (L/qL)^\ell\setminus\bigcup_{i=1}^n \bar{H}_i\\
&=\set{ \bar{x} \in L/qL \mid \alpha_i(x) \not\equiv 0 \mod{q}\quad  \text{for all $i \in \{1,\ldots,n\}$}}.
\end{split}
\end{align*}
It is known 
({\cite[Theorem 2.4]{ktt-cent}} and {\cite[Theorem 3.1]{ktt-noncentral}} for non-central case) that 
\begin{align*}
\chi_{\text{quasi}}(\mathcal{A};\,q) \ceq \#M(\mathcal{A};\,q)
\end{align*}
is a quasi-polynomial. 
\end{example}

Roughly speaking, the notion of characteristic quasi-polynomial is 
a $\mmoodd{q}$-version of the Ehrhart quasi polynomial. 
However, the characteristic quasi-polynomials possess some additional
properties. First, 
the constituents $g_r(t)$ ($r \in \{1, \dots, \tilde{n}\}$) 
of the characteristic quasi-polynomial 
$\chi_{\text{quasi}}(\mathcal{A};\,q)$ satisfy the 
$\gcd$-property. Secondly, the first constituent 
$g_1(t)$ (and equivalently, $g_r(t)$ for $r$ coprime to $\tilde{n}$) 
is known to be equal to the characteristic polynomial 
$\chi(\mathcal{A}, t)$ of the arrangement $\mathcal{A}$ (see \cite{ot}).
Furthermore, $g_{\tilde{n}}(t)$ is the characteristic polynomial of 
the associated toric arrangement \cite{lty, ty}.
The characteristic quasi-polynomial is an important concept, not only 
in the context of enumerative problems but also in its connections with 
arithmetic matroids and toric arrangements \cite{dm-ari, lty}.

\subsection{Equivariant Ehrhart theory}

In \cite{Stapledon}, Stapledon proposed 
an equivariant version of Ehrhart theory. 
Let $L \simeq \mathbb{Z}^n$ be a lattice and let $\G$ be a finite group acting 
linearly on $L$ via $\rho : \G \to \GL(L)$. 
Let $\mathcal{P}$ be a $\G$-invariant lattice polytope. 
For a positive integer $q\in\mathbb{Z}_{>0}$, the group $\G$ acts on 
the lattice points $q\mathcal{P} \cap L$. 
Let $\chi_{q\mathcal{P}}$ denote the character of this permutation representation.
Stapledon proved representation-theoretic analogues of 
several classical results in Ehrhart theory. 
For example, the map 
\[
F: \mathbb{Z}_{>0}\longrightarrow R(\G),\quad q\longmapsto 
F(q)=\chi_{q\mathcal{P}}
\]
is a quasi-polynomial of degree $\dim\mathcal{P}$ with the 
leading coefficient 
$\frac{\vol{\mathcal{P}}}{\#\G}\chi_{\st}$, where $\chi_{\st}$ is the regular (standard) character of $\G$ 
(\cite[Theorem 5.7, Corollary 5.9]{Stapledon}). 
It was also proved that the number 
of $\G$-orbits in $q\mathcal{P} \cap L$ is 
a quasi-polynomial in $q$. 

Stapledon also proved the following reciprocity. 
Let $F^\ast(q)=\chi_{q\mathcal{P*}}$ be the permutation representation 
of the lattice points in the interior of $q\mathcal{P}$. 
Then, from the Ehrhart reciprocity, the relation 
\begin{equation}
\label{eq:equivrecip}
F^*(q)=(-1)^{\dim\mathcal{P}}\det(\rho) F(-q)
\end{equation}
holds. 
Note that when $\G$ is the trivial group, these results recover 
the classical results in Ehrhart theory. 

\subsection{Towards an equivariant version of characteristic quasi-polynomials}

It is natural to consider $\mmoodd{q}$-version of the equivariant Ehrhart quasi-polynomial, 
namely, the equivariant characteristic quasi-polynomials for an arrangement invariant 
under a group action. For simplicity, in this paper, we do not consider hyperplanes, 
and instead focus solely on the $\mmoodd{q}$ permutation representation. 
Let $\G$ be a finite group and let $L \simeq \mathbb{Z}^\ell$ be a lattice.
Suppose that $\G$ acts linearly on $L$.
Then the action of $\G$ on $\X_q \ceq L/qL$ is naturally induced for each $q \in \mathbb{Z}_{>0}$. 
One of our problems is how the permutation character $\chi_{\X_q}$ of $L_q$ depends on $q$.
In the case where $\G$ is the Weyl group and $L$ is a lattice associated with a root system, 
there are several known results (e.g. {\cite{hai, rho}}), especially for $q \equiv 1 \mod{h}$, where $h$ is the Coxeter number. 
In this paper, we consider a general finite group action on a lattice. 
The main result of this paper is the following. 
\begin{theorem}[Theorem \ref{thm:main} below]
\label{Main result}
Consider the function $F : \mathbb{Z}_{>0} \longrightarrow R(\G)$ defined by $q \longmapsto \chi_{\X_q}$.
Then $F$ is a quasi-polynomial of degree $\ell$ ($ = \rank {L}$) with the $\gcd$-property.
\end{theorem}
We will also prove that the number of 
$\G$-orbits is a  quasi-polynomial 
with the $\gcd$-property (Corollary \ref{cor:G-orbit}). 

Let $\chi_1, \dots, \chi_k$ be irreducible characters of $\G$. Then $\chi_{\X_q}$ 
can be expressed as 
\begin{equation}
\label{eq:lincomb}
    \chi_{\X_q}=
m(\chi_1;\, q)\cdot\chi_1
+\cdots+
m(\chi_k;\, q)\cdot\chi_k, 
\end{equation}
with $m(\chi_k;\, q)\in\mathbb{Z}$. Theorem \ref{Main result} is equivalent 
to that each 
$m(\chi_i;\, q)$ is a $\mathbb{Z}$-valued quasi-polynomial in $q$ 
(Corollary \ref{multi}). 

There are several relations among these quasi-polynomials. 
In particular, there is a reciprocity-type relation between 
$m(\chi_i;\, q)$ and $m(\chi_i\otimes\delta_\rho;\, q)$, 
where $\delta_\rho=\det\rho$. More precisely, we have (Theorem \ref{thm:recip})
\[
m(\chi_i\otimes\delta_\rho;\, q)=(-1)^\ell m(\chi_i;\, -q).
\]
This implies the following relation (Corollary \ref{cor:rec}): 
\begin{equation}
\label{eq:selfdual}
F(-q)=(-1)^{\ell}\delta_\rho F(q). 
\end{equation}
Although the formula (\ref{eq:selfdual}) appears similar to 
(\ref{eq:equivrecip}), they are different in nature. 
It is important to note that (\ref{eq:equivrecip}) represents a reciprocity 
between $F(q)$ and $F^*(q)$, whereas (\ref{eq:selfdual}) is 
a self-duality of $F(q)$. 

We will also provide several explicit examples. 

\section{Quasi-polynomiality}

\subsection{Group action and representation}

We recall several notions and basic facts 
about representations of finite groups \cite{serre}. 

Let $\G$ be a finite group.
Let $V$ be a finite-dimensional vector space over $\mathbb{C}$, and 
let $\GL(V)$ denote the group of linear isomorphisms of $V$ onto itself.
A \textit{(linear) representation} of $\G$ on $V$ is a homomorphism $\rho:\G \longrightarrow \GL(V)$.
In this paper, we assume that $\rho$ is injective. 
The space $V$ is called the \textit{representation space} of $\rho$.

The \textit{character} $\chi_\rho : \G \longrightarrow \mathbb{C}$ of the representation $\rho$ is 
the function defined by $\g \longmapsto \tr\rho(\g)$, where $\tr$ denotes the trace function.
The character $\chi_\rho$ is constant on each conjugacy class. 
A function $\phi:\G\longrightarrow \mathbb{C}$ is called a class function 
if $\phi$ is constant on each conjugacy class. 
For functions $\phi,\psi: \G\longrightarrow\mathbb{C}$, define the inner product $(\phi, \psi)$ by 
\begin{align*}
(\phi,\psi) = \dfrac{1}{\#\G}\sum_{\g \in \G}\phi(\g)\overline{\psi(\g)},
\end{align*}
where $\bar{z}$ denotes the complex conjugate of $z \in \mathbb{C}$. 
Let $\chi_1, \dots, \chi_k$ be the set of all irreducible characters of $\G$. 
Then $\chi_1, \dots, \chi_k$ form an orthonormal basis of the space of 
class functions. In particular, $(\chi_i, \chi_j)=\delta_{ij}$. 
Thus, if a class function $\chi$ is expressed as a linear combination 
of irreducible characters $\chi=m_1\chi_1+\dots+m_k\chi_k$, then we have 
$m_i=(\chi, \chi_i)$. 

Let $\subG$ be a subgroup of $\G$. The restriction of 
a class function $\chi: \G\longrightarrow\mathbb{C}$ to $\subG$ is clearly 
a class function on $\subG$, which is denoted by 
$\Res^\G_\subG \chi : \subG \longrightarrow \mathbb{C}$. Conversely, 
for a class function $\varphi:\subG \longrightarrow \mathbb{C}$, 
define the \textit{induced function} $\Ind^\G_\subG \varphi:\G \longrightarrow \mathbb{C}$ 
by
\begin{equation}
    \left(\Ind^\G_\subG\varphi\right)(\g) = \dfrac{1}{\#\subG}\sum_{\substack{\eta \in \G\\\eta^{-1}\gamma\eta \in \subG}}\varphi(\eta^{-1}\gamma\eta).
\end{equation}
These two operators are related by the following Frobenius reciprocity: 
\begin{equation}
\label{FrobRecip}
\left(\chi,\, \Ind^\G_\subG\varphi\right) = \left(\Res^\G_\subG\chi,\, \varphi\right).
\end{equation}

Recall that the \textit{representation ring} $R(\G)$ of $\G$ is $\bigoplus_{V}\mathbb{Z}[V]/{\sim}$, where 
$V$ runs all finite-dimensional representations of $\G$, and $\sim$ is an equivalence 
relation generated by $[V]\sim [V']$ for isomorphic representations $V\simeq V'$ 
and $[V_1\oplus V_2]\sim [V_1]+[V_2]$. The multiplication is defined by 
$[V_1]\cdot[V_2]=[V_1\otimes V_2]$. The character gives a natural isomorphism 
of abelian groups 
\[
R(\G) \simeq \mathbb{Z}\chi_1 \oplus \cdots \oplus \mathbb{Z}\chi_k.
\]
The trivial representation $\rho_{\bm{1}}$ is the unit element in 
$R(\G)$. The character of $\rho_{\bm{1}}$ is denoted by $\bm{1}$.

Suppose that $\G$ acts on a finite set $X$.
Let $\mathbb{C}X$ denote the vector space based on $X$, that is, $\mathbb{C}X = \bigoplus_{x \in X}\mathbb{C}x$.
This gives rise to a natural representation $\rho_X : \G \longrightarrow \GL(\mathbb{C}X)$, 
which is called the \textit{permutation representation} of $X$. 
In the case of $X = \G$ with action defined by the left multiplication, 
it is called the \textit{regular (standard) representation}, 
denoted by $\rho_{\st}$. 
Note that the character $\chi_{\st}$ of the regular representation 
satisfies the following. 
\begin{align*}
\chi_{\st} = \sum_{i=1}^k\chi_i(1)\chi_i,\qquad \chi_\st(\g) = \begin{cases*}
\#\G & if $\g = 1$;\,\\
0 & otherwise.
\end{cases*}
\end{align*}

For $x \in X$, the \textit{$\G$-orbit} $\Gx$ 
and the \textit{isotropy subgroup} $\G_x$ 
are defined as follows: 
\[
\begin{split}
\Gx &= \set{\g x \in X \mid \g \in \G},\\
\G_x &= \set{\g \in \G \mid \g x = x}. 
\end{split}
\]

\subsection{Multiplicities of irreducible decompositions}

Let $L$ be a lattice, and $\{\beta_1,\ldots,\beta_\ell\}$ be a $\mathbb{Z}$-basis of $L$, 
that is, 
$L = \mathbb{Z}\beta_1 \oplus \cdots \oplus \mathbb{Z}\beta_\ell \simeq \mathbb{Z}^\ell$.
We identify an element $x = x_1\beta_1 + \cdots + x_\ell\beta_\ell$ of $L$ with the row vector $x = (x_1,\ldots,x_\ell)$ of $\mathbb{Z}^\ell$.

Let $\G$ be a finite group. 
Let $\rho : \G \longrightarrow \GL(L)$ be a group homomorphism. 
Let us denote the representation matrix of $\rho(\gamma)$ by 
$R_\gamma$, and we consider the right multiplication, namely, 
\begin{align*}
\rho(\g) : L\longrightarrow L,\quad x \longmapsto xR_\g. 
\end{align*}

For $q \in \mathbb{Z}_{>0}$, define $\mathbb{Z}_q \ceq \mathbb{Z}/q\mathbb{Z}$.
We will consider the following \textit{$q$-reduction} of $x = (x_1,\ldots,x_\ell) \in \mathbb{Z}^\ell$:
\begin{align*}
[x]_q \ceq ([x_1]_q,\ldots,[x_\ell]_q) \in \mathbb{Z}_q^\ell,
\end{align*}
where $[x_i]_q = x_i + q\mathbb{Z} \in \mathbb{Z}_q$.
We similarly consider the $q$-reduction of an integer matrix $A = (a_{ij})_{ij}$: 
\begin{align*}
[A]_q \ceq \left([a_{ij}]_q\right)_{ij}.
\end{align*}
Let $\varphi :\mathbb{Z}^\ell \longrightarrow \mathbb{Z}^\ell$ be a $\mathbb{Z}$-homomorphism represented by a $\ell \times \ell$ integer matrix $A$.
We can define the induced morphism $\varphi_q : \mathbb{Z}_q^\ell \longrightarrow \mathbb{Z}_q^\ell$ by
\begin{align*}
x \longmapsto x[A]_q.
\end{align*}

Let $\X_q \ceq L/qL \simeq (\mathbb{Z}/q\mathbb{Z})^\ell$.
The action of $\G$ on $\X_q$ is induced by $\rho(\g)_q : \X_q \to \X_q$.
Let $\chi_{\X_q}$ denote the character of the permutation representation of $\X_q$, and consider 
its irreducible decomposition:
\begin{align*}
\chi_{\X_q} = m(\chi_1;\,q) \cdot \chi_1 + \cdots + m(\chi_k;\,q) \cdot \chi_k,
\end{align*}
where $m(\chi_i;\,q)$ denotes the multiplicity of $\chi_i$ in $\chi_{\X_q}$.
Since $\chi_{\X_q}(\g)$ is equal to the number of elements in $\X_q$ fixed by $\g \in \G$, we have
\begin{align}
m(\chi_i;\,q) = (\chi_i,\chi_{\X_q}) = \dfrac{1}{\#\G}\sum_{\g \in \G}\chi_i(\g)\overline{\chi_{\X_q}(\g)} = \dfrac{1}{\#\G}\sum_{\g \in \G}\chi_i(\g) \cdot \#\X_q^\g, \label{m(i,q)}
\end{align}
where $\X_q^\g \ceq \set{x \in \X_q \mid \g x = x}$.
Thus, by studying the properties of $\#\X_q^\g$, we can determine how $m(\chi_i;\,q)$ depends on $q$.
Note that for the trivial character $\bm{1}$, the multiplicity $m(\bm{1};\, q)$ represents the number of $\G$-orbits of $\X_q$, according to Burnside's lemma. 

The fixed point set $\X_q^\g$ is expressed as 
\begin{align*}
\X_q^\g
 &= \set{x \in \X_q \mid \g x = x}\\
 &= \set{x \in \X_q \mid x[R_\g]_q = x}\\
 &= \set{x \in \X_q \mid x[R_\g - I_\ell]_q = 0},
\end{align*}
where $I_\ell$ is the identity matrix of size $\ell$.
Therefore, $\X_q^\g$ is equal to the kernel of the induced morphism 
$\left( \rho(\g) - \id\right)_q$. 
The cardinality of the kernel of this induced morphism is known to be quasi-monomial, 
as shown in \cite{ktt-cent}: 

\begin{lemma}[{\cite[Lemma 2.1]{ktt-cent}}]
Let $\varphi : \mathbb{Z}^\ell \longrightarrow \mathbb{Z}^\ell$ be a $\mathbb{Z}$-homomorphism.
Then The cardinality of the kernel of this induced morphism $\varphi_q : \mathbb{Z}_q^\ell \longrightarrow \mathbb{Z}_q^\ell$ is a quasi-monomial in $q$.
Furthermore, suppose $\varphi$ is represented by a matrix $A$.
Then the quasi-monomial $\#\ker{\varphi_q}$ can be expressed as
\begin{align}
\#\ker{\varphi_q} = \left(\prod_{j=1}^r \gcd\{e_j,\, q\} \right) q^{\ell - r}, \label{sl2.1}
\end{align}
where $r \ceq \rank{A}$ and $e_1,\ldots,e_r \in \mathbb{Z}_{>0}$, with $e_1 \mid e_2 \mid \cdots \mid e_r$, 
are the elementary divisors of $A$.
Hence, the quasi-monomial $\#\ker{\varphi_q}$ has the $\gcd$-property and the minimum period $e_r$.
If $r = 0$, we consider $e_0$ to be $1$.
\end{lemma}
\begin{proof}
Here, we only review quasi-monomiality.
For further details, see {\cite[Lemma 2.1]{ktt-cent}}.

Since $\#\ker{\varphi_q} = q^\ell / \#\image{\varphi_q}$, we will study $\#\image{\varphi_q}$.
Consider the Smith normal form
\begin{align*}
SAT = \begin{pmatrix}
e_1\\
& \ddots\\
&& e_r\\ 
&&& O
\end{pmatrix},\qquad r = \rank{A}, \quad e_1,\ldots,e_r,\in \mathbb{Z}_{>0}, \quad e_1 \mid e_2 \mid \cdots \mid e_r,
\end{align*}
where $S$ and $T$ are $\ell \times \ell$ unimodular matrices.
Since unimodularity is preserved under $q$-reductions, we may assume that $A$ is a diagonal matrix $\diag(e_1,\ldots,e_r,0,\ldots,0)$ from the outset.
Then, for $x = (x_1,\ldots,x_\ell) \in \mathbb{Z}_q^\ell$, we have
\begin{align*}
\varphi_q(x) = \left( [e_1]_qx_1,\,\ldots,\, [e_r]_qx_r,\, 0,\,\ldots,\, 0\right)
\end{align*}
and hence $\image{\varphi_q} = [e_1]_q\mathbb{Z}_q \times \cdots \times [e_r]_q\mathbb{Z}_q$.
Therefore, 
\begin{align*}
\#\image{\varphi_q} = \dfrac{q^r}{\prod_{j=1}^r \gcd\{e_j,\, q\}},
\end{align*}
and we obtain \eqref{sl2.1}.
\end{proof}

\begin{corollary}\label{multi}
The multiplicity $m(\chi_i;\,q)$ of $\chi_i$ in $\chi_{\X_q}$ is a quasi-polynomial in $q$. More explicitly, 
\begin{align}
m(\chi_i;\,q)  = \dfrac{1}{\#\G} \sum_{\g \in \G}\chi_i(\g) \cdot \left(\prod_{j=1}^{r(\g)} \gcd\{e_{\g,j},\, q\} \right) q^{\ell - r(\g)}, \label{st2.2}
\end{align}
where $r(\g) \ceq \rank{(R_\g - I_\ell)}$ and $e_{\g,1},\ldots,e_{\g,r(\g)} \in \mathbb{Z}_{>0}$ with $e_{\g,1} \mid e_{\g,2} \mid \cdots \mid e_{\g,r(\g)}$, are the elementary divisors of $R_\g-I_\ell$.
\end{corollary}
\begin{proof}
The equation \eqref{st2.2} is given by \eqref{m(i,q)} and \eqref{sl2.1}.
\end{proof}

Next, we present some properties of $m(\chi_i;\,q)$.

\begin{proposition}
The quasi-polynomial $m(\chi_i;\,q)$ has the $\gcd$-property with a period 
\begin{align*}
\tilde{n} \ceq \lcm\set{e_{\g,r(\g)} \mid \g \in \G}.
\end{align*}
Furthermore, the minimum period of the quasi-polynomial $m(\bm{1};\, q)$ is equal to $\tilde{n}$. 
\end{proposition}
Note that if $\chi_i \neq \bm{1}$, we do not know if $\tilde{n}$ is the minimum period.
\begin{proof}
Let $\g \in \G$ be an element that is not the identity, and let $e_{\g,1},\ldots,e_{\g,r(\g)}$ be the elementary divisors of $R_\g-I_\ell$.
Since $e_{\g,j}$ divides $\tilde{n}$ for $j \in \{1,\ldots,r(\g)\}$, we have
\begin{align*}
\prod_{j=1}^{r(\g)} \gcd\{e_{\g,j},\,q\} 
 = \prod_{j=1}^{r(\g)} \gcd\{e_{\g,j},\, \tilde{n},\, q\} 
 = \prod_{j=1}^{r(\g)} \gcd\left\{e_{\g,j},\, \gcd\{\tilde{n},\, q\}\right\}.
\end{align*}
Hence $m(\chi_i;\,q)$ depends on $q$ only through $\gcd\{\tilde{n},\, q\}$, which means that 
$\tilde{n}$ is a period of $m(\chi_i;\, q)$.

Let $g_{1}(t),\ldots, g_{\tilde{n}}(t) \in \mathbb{Z}[t]$ denote the constituents of the quasi-polynomial $m(\bm{1};\, q)$.
Since $\tilde{n}$ is divisible by $e_{\g,r(\g)}$ for all $\g \in \G$, we have
\begin{align}
    g_{\tilde{n}}(t) = \dfrac{1}{\#\G}\sum_{\g \in \G}\left(\prod_{j = 1}^{r(\g)}e_{\g,j}\right) t^{\ell - r(\g)} \label{gtilden}
\end{align}
from equation \eqref{st2.2}.
Suppose that $r < \tilde{n}$.
Then there exists $\g \in \G$ such that $\gcd\{e_{\g,r(\g)},r\} \neq e_{\g,r(\g)}$.
Since $\gcd\{e_{\g,j},r\} \leq e_{\g,j}$ for any $\g \in \G$ and $j \in \{1,\ldots,r(\g)\}$, we conclude that
\begin{align*}
    g_r(t) = \dfrac{1}{\#\G} \sum_{\g \in \G}\left(
    \prod_{j=1}^{r(\g)}\gcd\{e_{\g,j},q\}    \right) t^{\ell-r(\g)} \neq g_{\tilde{n}}(t)
\end{align*}
by the equations \eqref{st2.2} and \eqref{gtilden}, and hence $r$ is not a period.
Therefore, $\tilde{n}$ is the minimum period of $m(\bm{1};\, q)$.
\end{proof}

\begin{proposition}\label{p2.4}
The leading term of the quasi-polynomial 
$m(\chi_i;\,q)$ is $\frac{\chi_i(1)}{\#\G}q^\ell$.
\end{proposition}
\begin{proof}
Since $\rho$ is injective, $r(\g)=0$ holds if and only if $\g$ is the 
identity. 
Therefore, by Corollary \ref{multi}, the leading term of $m(\chi_i;\,q)$ is $q^\ell$ with the coefficient $\frac{\chi_i(1)}{\#\G}$.
\end{proof}

\subsection{Permutation representations}

Since each multiplicity $m(\chi_i;\,q)$ is a quasi-polynomial, the quasi-polynomiality of the function $F : q \longmapsto \chi_{\X_q}$ follows immediately.
The following theorem is the main result of this paper.
\begin{theorem}[Restatement of Theorem \ref{Main result}]
\label{thm:main}
Consider the function $F : \mathbb{Z}_{>0} \longrightarrow R(\G)$ defined by $q \longmapsto \chi_{\X_q}$.
Then $F$ is a quasi-polynomial of degree $\ell$. Furthermore, 
$F$ has the $\gcd$-property, the minimum period $\tilde{n}$, and 
the leading coefficient of the quasi-polynomial $\chi_{\X_q}$ is $\frac{\chi_\st}{\#\G}$.
\end{theorem}
\begin{proof}
By equation \eqref{st2.2}, we have
\begin{align*}
F(q) = \chi_{\X_q} = \sum_{i=1}^km(\chi_i;\, q) \cdot \chi_i = \dfrac{1}{\#\G}\sum_{i=1}^k\sum_{\g \in \G} \chi_i(\g) \cdot \left( \prod_{j=1}^{r(\g)} \gcd\{e_{\g,j},\, q\} \right) \cdot \chi_i \cdot  q^{\ell - r(\g)} \in R(\G)[q],
\end{align*}
hence $F$ is a quasi-polynomial with the $\gcd$-property.
Since $m(\bm{1};\, q)$ has the minimum period $\tilde{n}$, $F$ also has the minimum period $\tilde{n}$.

By Proposition \ref{p2.4}, the leading term of each multiplicity $m(\chi_i;\, q)$ is $\frac{\chi_i(1)}{\#\G}q^\ell$.
Thus, we have
\begin{align*}
\sum_{i=1}^k\dfrac{\chi_i(1)}{\#\G} \cdot \chi_i \cdot q^\ell = \dfrac{\chi_\st}{\#\G}q^\ell
\end{align*}
as the leading term of $F$.
\end{proof}

\subsection{Number of orbits}

In this section, we prove the quasi-polynomiality of the number of $\G$-orbits.
First, we describe the permutation character $\chi_{\Gx}$ on the $\G$-orbit $\Gx$ of $x \in L_q$.
\begin{lemma}\label{G-orbit_lemma}
    Let $\Gx$ denote the $\G$-orbit of $x \in L_q$.
    Then we have
    \begin{align*}
        \chi_{\Gx}(\g)
        = \#\G(x)^\g
        = \left(\Ind^\G_{\G_x}\bm{1}\right)(\g).
    \end{align*}
\end{lemma}
\begin{proof}
    An element $\eta x$ of $\G(x)$ is fixed by $\g$ if and only if $\eta^{-1}\g\eta$ fixes $x$.
    Thus, the cardinality of $\Gx^\g$ is 
    \begin{align*}
        \#\Gx^\g
        = \dfrac{\#\set{\eta \in \G \mid \eta^{-1}\g\eta \in \G_x}}{\#\G_x}.
    \end{align*}
    On the other hand, it follows directly that the above expression is equal to $\left(\Ind^\G_{\G_x}\bm{1}\right)(\g)$:
    \begin{align*}
        \left(\Ind^\G_{\G_x}\bm{1}\right)(\g)
        = \dfrac{1}{\#\G_x}\sum_{\substack{\eta \in \G\\\eta^{-1}\g\eta \in \G_x}}\bm{1}(\eta^{-1}\g\eta)
        = \dfrac{\#\set{\eta \in \G \mid \eta^{-1}\g\eta \in \G_x}}{\#\G_x}.
    \end{align*}
\end{proof}

For a 1-dimensional character $\lambda$  of $\G$ and $q \in \mathbb{Z}_{>0}$, let $f_{L/\GAMMA}(\lambda;\, q)$ denote the number of $\G$-orbit of $\X_q$ whose isotropy subgroup is contained in the subgroup $\lambda^{-1}(1)$ of $\G$.
Using the Frobenius reciprocity (\ref{FrobRecip}), we obtain the following lemma.
\begin{lemma}\label{numoforbit}
    For a 1-dimensional character $\lambda$  of $\G$ and $q \in \mathbb{Z}_{>0}$, we have
    \begin{align*}
        f_{L/\G}(\lambda;\, q) = (\lambda, \chi_{\X_q}) = m(\lambda;\, q).
    \end{align*}
    
\end{lemma}
\begin{proof}
    Note that the second equality is the definition of $m(\lambda;\, q)$.
    
    Note that the permutation character $\chi_{\X_q}$ can be decomposed into a sum of all permutation characters of $\G$-orbit of $L_q$:
    \begin{align*}
        \chi_{\X_q} = \sum_{\text{$\Gx$ : $\G$-orbit}}\chi_{\Gx}.
    \end{align*}
    By Lemma \ref{G-orbit_lemma} and Frobenius reciprocity (\ref{FrobRecip}), we have
    \begin{align*}
        (\lambda,\chi_{\X_q})
        = \sum_{\text{$\Gx$ : $\G$-orbit}}(\lambda, \chi_{\Gx})
        = \sum_{\text{$\Gx$ : $\G$-orbit}}\left(\lambda,\, \Ind^\G_{\G_x}\bm{1}\right)
        = \sum_{\text{$\Gx$ : $\G$-orbit}}\left(\Res^\G_{\G_x}\lambda,\, \bm{1}\right).
    \end{align*}
    Since $\Res^\G_{\G_x}\lambda$ is an irreducible character of $\G_x$,
    the orthogonality of irreducible characters implies that
    \begin{align*}
        \left(\Res^\G_{\G_x}\lambda,\, \bm{1}\right)
        &= \begin{cases}
            1 & \G_x \subseteq \lambda^{-1}(1);\\
            0 & \text{otherwise}.
        \end{cases}
    \end{align*}
    Therefore, we have $(\lambda,\chi_{\X_q}) = f_{L/\G}(\lambda;\, q)$.
\end{proof}

\begin{corollary}
\label{cor:G-orbit}
    The function $f_{L/\GAMMA}(\lambda;\,-) : \mathbb{Z}_{>0} \longrightarrow \mathbb{Z}$ is a quasi-polynomial of degree $\ell$ and it has the $\gcd$-property.
\end{corollary}
\begin{proof}
    This follows from Corollary \ref{multi} and Lemma \ref{numoforbit}.
\end{proof}

\subsection{Reciprocity for the multiplicities}

Let $\rho : \G \longrightarrow \GL(L)$ be a representation and $R_\g$ the representation matrix of $\rho(\g)$.
Define the \textit{reciprocity character} $\delta_\rho : \G \longrightarrow \mathbb{C}$ by
\begin{align*}
    \delta_\rho(\g) \ceq (-1)^{r(\g)},
\end{align*}
where $r(\g) = \rank(R_\gamma - I_\ell)$.
The following lemma shows that $\delta_\rho(\g) = \det{R_\g}$ and that 
$\delta_\rho$ is an irreducible character of $\G$.

\begin{lemma}[{\cite[Lemma 5.5]{Stapledon}}]
\label{lem:det}
Let $R\in\GL_n(\mathbb{R})$ be a real matrix of finite order. 
Let $r \ceq \rank(R-I_n)$.
Then $\det{R} = (-1)^r$. 
\end{lemma}
\begin{proof}
Since $R$ is finite order, it is diagonalizable (in $\mathbb{C}$), and 
we can write $R=PDP^{-1}$, where $P, D\in\GL_n(\mathbb{C})$ with $D$ diagonal. 
Clearly, $\rank(R-I_n)=\rank (D-I_n)$. Thus, $r$ is the number 
of diagonal entries of $D$ that are not equal to $1$. Since $R$ is 
a real matrix, the set of eigenvalues is closed under complex conjugation. 
The finiteness of the order implies that all the eigenvalues have 
absolute value $1$. Therefore, the diagonal entries of $D$ are 
as follows (with multiplicities): 
\begin{align*}
1^{p_1},\ (-1)^{p_2},\ 
\alpha_1^{q_1},\ \overline{\alpha}_1^{q_1},\ 
\alpha_2^{q_2},\ \overline{\alpha}_2^{q_2},\ \dots ,\ 
\alpha_m^{q_m},\ \overline{\alpha}_m^{q_m},
\end{align*}
with $p_i, q_j\in\mathbb{Z}$ and $|\alpha_j|=1$. Hence, we have 
\begin{align*}
r=p_2+2(q_1+q_2+\dots+q_m), 
\end{align*}
and $\det{D} = (-1)^{p_2}$. Thus, $\det{R} = (-1)^r$. 
\end{proof}

The quasi-polynomials 
$m(\chi_i \otimes \delta_\rho;\, q)$ and $m(\chi_i;\, q)$ are connected 
by the following formula. 
\begin{theorem}[Reciprocity theorem]
\label{thm:recip}
The following formula holds for an irreducible character $\chi_i$ of $\G$: 
\begin{equation}
\label{eq:recip}
m(\chi_i \otimes \delta_\rho;\, q) = (-1)^\ell m(\chi_i;\, -q). 
\end{equation}  
\end{theorem}
\begin{proof}
Note that since $\chi_i$ is an irreducible character, 
$\chi_i \otimes \delta_\rho$  is also irreducible. 
Using (\ref{st2.2}), we compute $m(\chi_i\otimes\delta_\rho;\, q)$ 
as follows (omitting $\prod_{j=1}^{r(\g)}$): 
\begin{align*}
(\chi_i \otimes \delta_\rho)(\g) \cdot q^{\ell-r(\g)} = \chi_i(\g)(-1)^{r(\g)} q^{\ell-r(\g)} = \chi_i(\g)(-1)^\ell (-q)^{\ell-r(\g)}
\end{align*}
for each $\g \in \G$. This implies (\ref{eq:recip}). 
\end{proof}
Note that the map $F(q)=\chi_{X_q}$ can be extended to 
$F:\mathbb{Z}\longrightarrow R(\G)$ as a quasi-polynomial. 
\begin{corollary}
\label{cor:rec}
The quasi-polynomial $F:\mathbb{Z}\longrightarrow R(\G)$ satisfies 
$F(q) = (-1)^\ell\delta_\rho F(-q)$.
\end{corollary}
\begin{proof}
    By Theorem $\ref{thm:recip}$, it follows that 
    \begin{align*}
        F(q) = \chi_{\X_q} &= \sum_{i=1}^k m(\chi_i;\, q) \cdot \chi_i\\
        &= \sum_{i=1}^k (-1)^\ell m(\chi_i \otimes \delta_\rho;\,-q) \cdot \chi_i\\
        &= (-1)^\ell \sum_{i=1}^k m(\chi_i;\, -q) (\chi_i\otimes\delta_\rho)\\
        &= (-1)^\ell \delta_\rho F(-q).
        \end{align*}
\end{proof}

\subsection{Examples} 
We present some simple examples involving cyclic groups and symmetric groups.

\begin{example}\label{example_c6on2}
    Let $\G \ceq \mathbb{Z}/6\mathbb{Z}$ be a cyclic group of order $6$ generated by $\sigma$. 
    Let $\chi: \G \longrightarrow \mathbb{C}$ be the function that 
    sends $\sigma$ to $\zeta_6 \ceq e^{\frac{2\pi\sqrt{-1}}{6}}$. Then the irreducible characters of $\G$ are $\{\chi,\ldots,\chi^5,\chi^6 = \bm{1}\}$, where $\bm{1}$ is the character of the trivial representation of $\G$.
    Consider the action of $\G$ on $L \ceq \mathbb{Z}^2$ given by
    \begin{align*}
        \sigma \longmapsto R_\sigma \ceq \begin{pmatrix}
            0 & 1\\ -1 & 1
        \end{pmatrix}.
    \end{align*}
    To compute $\chi_{\X_q}$, we need to compute the rank and the elementary divisors of $R_{\sigma^i} - I_\ell$ for each $i \in \{1,\ldots,5\}$. They are as follows:
    \begin{align*}
        r(\sigma^i) = 2\ \text{ for all $i \in \{1,\ldots,5\}$},\qquad 
        (e_{\sigma^1,1},e_{\sigma^1,2}) = (e_{\sigma^5,1},e_{\sigma^5,2}) = (1,1),\\
        (e_{\sigma^2,1},e_{\sigma^2,2}) = (e_{\sigma^4,1},e_{\sigma^4,2}) = (1,3),\qquad 
        (e_{\sigma^3,1},e_{\sigma^3,2}) = (2,2).
    \end{align*}
    Hence, we obtain the multiplicity $m(\chi^j;q)$ as follows: 
    \begin{align*}
        m(\chi^1;\,q) = m(\chi^5;\,q) &= \begin{cases}
            \, \dfrac{1}{6}(q^2-1) & \gcd\{6,q\} = 1;\vspace{2mm}\\
            \, \dfrac{1}{6}(q^2-4) & \gcd\{6,q\} = 2;\vspace{2mm}\\
            \, \dfrac{1}{6}(q^2-3) & \gcd\{6,q\} = 3;\vspace{2mm}\\
            \, \dfrac{1}{6}(q^2-6) & \gcd\{6,q\} = 6,
        \end{cases}\\[6pt]
        m(\chi^2;\,q) = m(\chi^4;\,q) &= \begin{cases}
            \, \dfrac{1}{6}(q^2-1) & \gcd\{6,q\} = 1;\vspace{2mm}\\
            \, \dfrac{1}{6}(q^2+2) & \gcd\{6,q\} = 2;\vspace{2mm}\\
            \, \dfrac{1}{6}(q^2-3) & \gcd\{6,q\} = 3;\vspace{2mm}\\
            \, \dfrac{1}{6}q^2 & \gcd\{6,q\} = 6,
        \end{cases}\\[6pt]
        m(\chi^3;\,q) &= \begin{cases}
            \, \dfrac{1}{6}(q^2-1) & \gcd\{6,q\} = 1;\vspace{2mm}\\
            \, \dfrac{1}{6}(q^2-4) & \gcd\{6,q\} = 2;\vspace{2mm}\\
            \, \dfrac{1}{6}(q^2+3) & \gcd\{6,q\} = 3;\vspace{2mm}\\
            \, \dfrac{1}{6}q^2 & \gcd\{6,q\} = 6,
        \end{cases}\\[6pt]
        m(\bm{1};\,q) &= \begin{cases}
            \, \dfrac{1}{6}(q^2+5) & \gcd\{6,q\} = 1;\vspace{2mm}\\
            \, \dfrac{1}{6}(q^2+8) & \gcd\{6,q\} = 2;\vspace{2mm}\\
            \, \dfrac{1}{6}(q^2+9) & \gcd\{6,q\} = 3;\vspace{2mm}\\
            \, \dfrac{1}{6}(q^2+12) & \gcd\{6,q\} = 6.
        \end{cases}
    \end{align*}
    In this case, since $\delta_\rho = \bm{1}$, it follows that $m(\chi^j;\, q) = m(\chi^j;\, -q)$ for $j \in \{1,\ldots,6\}$.
    
    We also obtain $\chi_{\X_q}$ as
    \begin{align*}
        \chi_{\X_q} = \begin{cases}
            \, \dfrac{1}{6}\Bigl(\chi_\st q^2 + 6(\bm{1}) - \chi_\st\Bigr) & \gcd\{6,q\} = 1;\vspace{2mm}\\
            \, \dfrac{1}{6}\Bigl(\chi_\st q^2 + 12(\bm{1}) + 6(\chi^2 + \chi^4) - 4\chi_\st\Bigr) & \gcd\{6,q\} = 2;\vspace{2mm}\\
            \, \dfrac{1}{6}\Bigl(\chi_\st q^2 + 12(\bm{1}) + 6\chi^3 - 3\chi_\st\Bigr) & \gcd\{6,q\} = 3;\vspace{2mm}\\
            \, \dfrac{1}{6}\Bigl(\chi_\st q^2 + 18(\bm{1}) + 6(\chi^2 + \chi^3 + \chi^4) - 6\chi_\st \Bigr)& \gcd\{6,q\} = 6,
        \end{cases}
    \end{align*}
    where $\chi_\st = \chi + \cdots + \chi^6$ is the regular character of $\G$.
\end{example}

\begin{example}\label{example_c6on3}
    As in the previous example, we consider the cyclic group $\G = \mathbb{Z}/6\mathbb{Z}$.
    The action of $\G$ on $L \ceq \mathbb{Z}^3$ is given by 
    \begin{align*}
        \sigma \longmapsto R_\sigma \ceq \begin{pmatrix}
            -1 & -1 & 0\\ 1& 0 & 0\\ 0 & 0 & -1
        \end{pmatrix}.
    \end{align*}
    By computing in the same way, we obtain the following:
    \begin{align*}
        r(\sigma^1) = r(\sigma^5) = 3,\qquad r(\sigma^2) = r(\sigma^4) = 2,\qquad  r(\sigma^3) = 1,\\
        (e_{\sigma^1,1},e_{\sigma^1,2},e_{\sigma^1,3}) = (e_{\sigma^5,1},e_{\sigma^5,2},e_{\sigma^5,3}) = (1,1,6),\\
        (e_{\sigma^2,1},e_{\sigma^2,2}) = (e_{\sigma^4,1},e_{\sigma^4,2}) = (1,3),\qquad 
        e_{\sigma^3,1} = 2,
    \end{align*}
    and
    \begin{align*}
        m(\chi^1;\,q) = m(\chi^5;\,q) &= \begin{cases}
            \, \dfrac{1}{6}(q^3 - q^2 - q + 1) & \gcd\{6,q\} = 1;\vspace{2mm}\\
            \, \dfrac{1}{6}(q^3 - 2q^2 - q + 2) & \gcd\{6,q\} = 2;\vspace{2mm}\\
            \, \dfrac{1}{6}(q^3 - q^2 - 3q + 3) & \gcd\{6,q\} = 3;\vspace{2mm}\\
            \, \dfrac{1}{6}(q^3 - 2q^2 - 3q + 6) & \gcd\{6,q\} = 6,
        \end{cases}\\[6pt]
        m(\chi^2;\,q) = m(\chi^4;\,q) &= \begin{cases}
            \, \dfrac{1}{6}(q^3 + q^2 - q - 1) & \gcd\{6,q\} = 1;\vspace{2mm}\\
            \, \dfrac{1}{6}(q^3 + 2q^2 - q - 2) & \gcd\{6,q\} = 2;\vspace{2mm}\\
            \, \dfrac{1}{6}(q^3 + q^2 - 3q - 3) & \gcd\{6,q\} = 3;\vspace{2mm}\\
            \, \dfrac{1}{6}(q^3 + 2q^2 - 3q - 6) & \gcd\{6,q\} = 6,
        \end{cases}\\[6pt]
        m(\chi^3;\,q) &= \begin{cases}
            \, \dfrac{1}{6}(q^3 - q^2 + 2q - 2) & \gcd\{6,q\} = 1;\vspace{2mm}\\
            \, \dfrac{1}{6}(q^3 - 2q^2 + 2q - 4) & \gcd\{6,q\} = 2;\vspace{2mm}\\
            \, \dfrac{1}{6}(q^3 - q^2 + 6q - 6) & \gcd\{6,q\} = 3;\vspace{2mm}\\
            \, \dfrac{1}{6}(q^3 - 2q^2 + 6q - 12) & \gcd\{6,q\} = 6,
        \end{cases}\\[6pt]
        m(\bm{1};\,q) &= \begin{cases}
            \, \dfrac{1}{6}(q^3 + q^2 + 2q + 2) & \gcd\{6,q\} = 1;\vspace{2mm}\\
            \, \dfrac{1}{6}(q^3 + 2q^2 + 2q + 4) & \gcd\{6,q\} = 2;\vspace{2mm}\\
            \, \dfrac{1}{6}(q^3 + q^2 + 6q + 6) & \gcd\{6,q\} = 3;\vspace{2mm}\\
            \, \dfrac{1}{6}(q^3 + 2q^2 + 6q + 12) & \gcd\{6,q\} = 6.
        \end{cases}
    \end{align*}
    In this case, $\delta_\rho = \chi^3$.
    Then, we have $m(\chi^1;\, q) = -m(\chi^4;\, -q)$ and $m(\chi^3;\, q) = -m(\bm{1};\, -q)$.
    
    We also obtain $\chi_{\X_q}$ as follows: 
    \begin{align*}
        \chi_{\X_q} = \begin{cases}
            \, \dfrac{1}{6}\Bigl(
                \chi_\st q^3
                + \bigl(\bm{1} - \chi_{15} + \chi_{24} - \chi^3 \bigr)q^2
                + \bigl(2(\bm{1}) - \chi_{15} - \chi_{24} +2 \chi^3 \bigr) q\\
                \qquad \qquad \qquad \qquad \qquad \qquad \qquad \quad
                - \bigl(2(\bm{1}) + \chi_{15}- \chi_{24} -2 \chi^3 \bigr)
            \Bigr) & \gcd\{6,q\} = 1;\vspace{2mm}\\
            \, \dfrac{1}{6}\Bigl(
                \chi_\st q^3
                + 2\bigl(\bm{1} - \chi_{15} + \chi_{24} - \chi^3 \bigr)q^2
                + \bigl(2(\bm{1}) - \chi_{15} - \chi_{24} +2 \chi^3 \bigr) q\\
                \qquad \qquad \qquad \qquad \qquad \qquad \qquad \quad
                - 2\bigl(2(\bm{1}) + \chi_{15}- \chi_{24} -2 \chi^3 \bigr)
            \Bigr) & \gcd\{6,q\} = 2;\vspace{2mm}\\
            \, \dfrac{1}{6}\Bigl(
                \chi_\st q^3
                + \bigl(\bm{1} - \chi_{15} + \chi_{24} - \chi^3 \bigr)q^2
                + 3\bigl(2(\bm{1}) - \chi_{15} - \chi_{24} +2 \chi^3 \bigr) q\\
                \qquad \qquad \qquad \qquad \qquad \qquad \qquad \quad
                - 3\bigl(2(\bm{1}) + \chi_{15}- \chi_{24} -2 \chi^3 \bigr)
            \Bigr) & \gcd\{6,q\} = 3;\vspace{2mm}\\
            \, \dfrac{1}{6}\Bigl(
                \chi_\st q^3
                + 2\bigl(\bm{1} - \chi_{15} + \chi_{24} - \chi^3 \bigr)q^2
                + 3\bigl(2(\bm{1}) - \chi_{15} - \chi_{24} +2 \chi^3 \bigr) q\\
                \qquad \qquad \qquad \qquad \qquad \qquad \qquad \quad
                - 6\bigl(2(\bm{1}) + \chi_{15}- \chi_{24} -2 \chi^3 \bigr)
            \Bigr) & \gcd\{6,q\} = 6,
        \end{cases}
    \end{align*}
    where $\chi_{15} \ceq \chi^1 + \chi^5$ and $\chi_{24} \ceq \chi^2 + \chi^4$. 
\end{example}

\begin{example}\label{example_s3on2}
    Let $\G \ceq \mathfrak{S}_3$ be the symmetric group of degree $3$, which is also the Weyl group of type $A_2$.
    The group $\G$ has three irreducible characters: the trivial character $\bm{1}$, the determinant character $\delta$ and the $2$-dimensional character $\chi$.
    Consider the lattice $L \ceq \mathbb{Z}(e_1 - e_2) \oplus \mathbb{Z}(e_2 - e_3)$.
    The group $\G$ acts on $L$ as a permutation of $\set{e_1,e_2,e_3}$.
    
    Note that we only need to calculate the rank and the elementary divisors for the representative of each conjugacy class.
    Choose the representatives $\tau \ceq (1\ 2)$ and $\sigma \ceq (1\ 2\ 3)$.
    The representation matrices are given by
    \begin{align*}
        R_\tau = \begin{pmatrix}
            -1 & 1\\ 0 & 1
        \end{pmatrix},\qquad R_\sigma = \begin{pmatrix}
            0 & -1\\ 1 & -1
        \end{pmatrix}
    \end{align*}
    Thus, we have 
    \begin{align*}
        r(\tau) = 1,\quad r(\sigma) = 2,\qquad e_{\tau,1} = 1,\quad (e_{\sigma,1},e_{\sigma,2}) = (1,3).
    \end{align*}
    Therefore, we obtain 
    \begin{align*}
        m(\bm{1};\,q) &= \begin{cases}
            \dfrac{1}{6}(q^2 + 3q + 2) & \gcd\{3,q\} = 1;\vspace{2mm}\\
            \dfrac{1}{6}(q^2 + 3q + 6) & \gcd\{3,q\} = 3,
        \end{cases}\\[6pt]
        m(\delta;\,q) &= \begin{cases}
            \dfrac{1}{6}(q^2 - 3q + 2) & \gcd\{3,q\} = 1;\vspace{2mm}\\
            \dfrac{1}{6}(q^2 - 3q + 6) & \gcd\{3,q\} = 3,
        \end{cases}\\[6pt]
        m(\chi;\,q) &= \begin{cases}
            \dfrac{1}{6}(2q^2 - 2) & \gcd\{3,q\} = 1;\vspace{2mm}\\
            \dfrac{1}{6}(2q^2 - 6) & \gcd\{3,q\} = 3.
        \end{cases}
    \end{align*}
    In this case, $\delta_\rho = \delta$.
    Hence, we have $m(\bm{1};\, q) = m(\delta;\, -q)$ and $m(\chi;\, q) = m(\chi;\, -q)$.

    We also obtain $\chi_{\X_q}$ as
    \begin{align*}
        \chi_{\X_q} = \begin{cases}
            \dfrac{1}{6}\Bigl(
                \chi_\st q^2 + 3(\bm{1} - \delta)q + 2(\bm{1}-\delta + \chi)
            \Bigr) & \gcd\{3,q\} = 1;\vspace{2mm}\\
            \dfrac{1}{6}\Bigl(
                \chi_\st q^2 + 3(\bm{1} - \delta)q + 6(\bm{1}-\delta + \chi)
            \Bigr) & \gcd\{3,q\} = 3,\vspace{2mm}\\
        \end{cases}  
    \end{align*}
    where $\chi_\st = \bm{1} + \delta + 2\chi$.
    As Haiman mentions in {\cite[\S 7.4]{hai}}, the multiplicity $m(\bm{1};\,q)$ is equal to the Ehrhart quasi-polynomial $\LL_{\overline{A_0}}(q) = \#(q\overline{A_0} \cap L)$ of the fundamental alcove $\overline{A_0}$ of type $A_2$.
\end{example}


\begin{thebibliography}{99}

\bibitem{BR:2007} M. Beck and S. Robins, ``Computing the Continuous Discretely'',  Undergraduate Texts in Mathematics, Springer, 2007.

\bibitem{dm-ari}
M. D'Adderio, L. Moci. 
Arithmetic matroids, the Tutte polynomial and toric arrangements. 
\emph{Adv. Math.} \textbf{232}, (2013) 335--67.

\bibitem{hai}
M. D. Haiman, 
Conjectures on the quotient ring by diagonal invariants, 
\emph{J. Algebraic Combin.} \textbf{3} (1994), 17--76.

\bibitem{ktt-cent}
H. Kamiya, A. Takemura, H. Terao, 
Periodicity of hyperplane arrangements with integral coefficients modulo positive integers, 
\emph{J. Algebraic Combin.} \textbf{27} (2008), no. 3, 317--330. 

\bibitem{ktt-noncentral}
H. Kamiya, A. Takemura, H. Terao, 
Periodicity of Non-Central Integral Arrangements Modulo Positive Integers, 
\emph{Ann. Comb.}  \textbf{15} (2011), no. 3, 449--464. 

\bibitem{lty}
Y. Liu, T. N. Tran, M. Yoshinaga, 
$G$-Tutte polynomials and abelian Lie group arrangements, 
\emph{Int. Math. Res. Not. IMRN} (2021), no. 1, 152--190.

\bibitem{ot} 
P. Orlik and H. Terao, Arrangements of hyperplanes, 
Grundlehren der Mathematischen Wissenschaften, 300. 
Springer-Verlag, Berlin, 1992. xviii+325 pp.

\bibitem{rho}
B. Rhoades, 
Parking structures: Fuss analogs,
\emph{J. Algebraic Combin.} \textbf{40} (2014), 417--473.

\bibitem{serre}
J.-P. Serre, 
Linear representations of finite groups. 
Springer, 1977. 

\bibitem{Stapledon}
A. Stapledon,
Equivariant Ehrhart theory,
\textit{Adv. Math.} \textbf{226} (2011), no. 4, 3622--3654.

\bibitem{ty}
T. N. Tran, M. Yoshinaga, 
Combinatorics of certain abelian Lie group arrangements and chromatic 
quasi-polynomials. 
\emph{J. Combin. Theory Ser. A} \textbf{165} (2019), 258--272.






\end{thebibliography}
\end{document}